\documentclass[11pt]{amsart}
\usepackage{times}
\usepackage[T1]{fontenc}
\usepackage{amssymb, amsthm, amsmath}
\usepackage[active]{srcltx}
\usepackage{hyperref}
\usepackage{setspace}

\DeclareMathOperator{\dcl}{dcl}

 \DeclareMathOperator{\dom}{dom}

\newtheorem*{nthm}{Theorem}

\newtheorem{theorem}{Theorem}[section]

\newtheorem{claim}[theorem]{Claim}

\newtheorem{cor}[theorem]{Corollary}

\newtheorem{fact}[theorem]{Fact}
\newtheorem{lemma}[theorem]{Lemma}

\newtheorem{proposition}[theorem]{Proposition}

\theoremstyle{definition}
\newtheorem{definition}[theorem]{Definition}
\newtheorem{defnn}[theorem]{Definition}
\newtheorem{example}[theorem]{Example}
\newtheorem{remark}[theorem]{Remark}

\newtheorem*{remark*}{Remark}

\newcommand{\Qq}{{\mathbb{Q}}}

\newcommand{\sub}{\subseteq}
\newcommand{\dpr}{\mbox{dp-rk}}
\newcommand{\al}{\alpha}

\newcommand{\CL}{{\mathcal L}}

\newcommand{\CN}{{\mathcal N}}

\newcommand{\CM}{{\mathcal M}}

\newcommand{\CI}{\mathcal{I}}
\newcommand{\CF}{{\mathcal F}}

\newcommand{\0}{\emptyset}

\newcommand{\fleq}{\preccurlyeq}
\renewcommand{\phi}{\varphi}

\long\def\symbolfootnote[#1]#2{\begingroup%
\def\thefootnote{\fnsymbol{footnote}}\footnote[#1]{#2}\endgroup}

\makeatletter

\def\Ind#1#2{#1\setbox0=\hbox{$#1x$}\kern\wd0\hbox to 0pt{\hss$#1\mid$\hss}
\lower.9\ht0\hbox to 0pt{\hss$#1\smile$\hss}\kern\wd0}

\def\Notind#1#2{#1\setbox0=\hbox{$#1x$}\kern\wd0\hbox to 0pt{\mathchardef
\nn=12854\hss$#1\nn$\kern1.4\wd0\hss}\hbox to
0pt{\hss$#1\mid$\hss}\lower.9\ht0 \hbox to
0pt{\hss$#1\smile$\hss}\kern\wd0}

\def\la{\langle}
\def\ra{\rangle}

\makeatother

\title{Fields interpretable in real closed valued fields and some expansions}

\date{\today}
\author{Assaf Hasson}
\address{Department of Mathematics, Ben Gurion University of the Negev, Be'er-Sheva 84105, Israel}
\email{hassonas@math.bgu.ac.il}

\author{Ya'acov Peterzil}
\address{Department of Mathematics, University of Haifa, Haifa, Israel}
\email{kobi@math.haifa.ac.il}

\thanks{The second author was partially supported by Israel Science Foundation grant number 290/19 }

\begin{document}

	\begin{abstract}
		Let $\CM=\la K;O\ra$ be a real closed valued field and let $k$ be its residue field. We prove that every interpretable field in  $\CM$ is definably isomorphic to either $K$, $K(\sqrt{-1})$, $k$, or $k(\sqrt{-1})$.	The same result holds when $K$ is a model of $T$, for  $T$
		a power bounded
		 o-minimal expansion of a real closed field,  and $O$ is a $T$-convex subring.
		
		The proof is direct and does not make use of known results about elimination of imaginaries in valued fields.
	\end{abstract}
	
		\maketitle	

\section{Introduction }
The goal of this work is to classify fields interpretable (i.e. fields which are given as definable quotients) in real closed valued fields or, more generally, in expansions of
power bounded
o-minimal structures by a proper $T$-convex subring.  This type of  classification of definable quotients originates in Poizat's influential model theoretic consideration  of the Borel-Tits theorem, \cite{PoiFields}
%{\bf I am not familiar with this history, is this clearly true?}. %%Analogous results were obtained in the o-minimal context (\cite{OtPePi}), and very recently for free groups (\cite{SklFields}).
A similar study of interepretable groups and fields, in the setting of (pure) algebraically closed valued fields, was carried out By Hrushovski and Rideau-Kikuchi in \cite{HrRid} using different techniques. \\

We prove (see Section \ref{sec T-convex} for the definition of $T_{conv}$):
\begin{nthm}\label{main}
	 Let $K$ be a real closed valued field, or more generally a model of $T_{conv}$, where $T$ is an o-minimal
	 power bounded
	 expansion of a real closed field. Let $k$ be its residue field.  If $\CF$ is a field  interpretable in $K$ then $\CF$ is definably isomorphic to either $K$, $K(\sqrt{-1})$, $k$, or $k(\sqrt{-1}))$.
\end{nthm}

This result extends a similar theorem of Bays and the second author,   \cite{BaysPet}, proved for   {\em definable} fields in real closed valued fields. The analogous problem of definable fields in  $\Qq_p$ was addressed in \cite{PilQp}, but its generalisation to interpretable fields remains open.

%To a certain degree, this could be attributed to the fact that up until the work carried out in \cite{HrMarRid} definable quotients %(known in model theory as imaginaries, or imaginary sorts) in $\Qq_p$ were not well understood.

The study of imaginaries in valued fields was suggested by Holly in \cite{HolEOI1} and first studied in depth and in full generality, in the setting of algebraically closed valued fields, by Haskell, Hrushovski and Macpherson in \cite{HaHrMac1} and \cite{HaHrMac2}. Elimination of imaginaries in real closed valued fields analogous to that of \cite{HaHrMac1} was proved by Mellor in \cite{Mellor}. Some results analogous to  \cite{HaHrMac2} for real closed valued fields were obtained by Ealy, Haskell and Marikova in  \cite{EHM}.

The work in \cite{HrRid} uses the theorems on elimination of imaginaries, stable domination and other structural results on algebraically closed valued fields, and accomplishes  considerably more than the classification of interpretable fields. We adopt a different approach circumventing elimination of imaginaries, and avoiding almost completely the so called geometric sorts.
%Our proof is direct and does not make use of the above mentioned works on elimination of imaginaries in valued fields.
In fact, our main result covers  expansions of real closed valued fields by analytic functions where no elimination of imaginaries results are currently available  (see \cite{HaHrMac3}). Our proof is based on the analysis of one dimensional (equivalently dp-minimal) subsets of the interpreted field $\CF$, and as such it borrows ideas from Johnson's work on fields of finite dp-rank (see for example \cite{johnsondpfiniteI} and \cite{johnsondpfiniteII}), as well as \cite{PePiSt}.

The outline of the proof is as follows: We identify $\CF$ with $X/E$ for a definable $X\sub K^n$  and a definable equivalence relation $E$. We then find a one-dimensional definable set $J\sub X$ intersecting infinitely many $E$-classes. After possibly shrinking $J$ we endow it with the structure of a weakly o-minimal structure $\CI$. Its universe $I=J/E$  is  our basic building block (see \cite{johnsondpfiniteI} for the similar notion of {\em a quasi-minimal sets} or \cite{PePiSt} for the notion of $\mathbb G$-minimal set). Next, we show that, after possibly shrinking $J$ further,  $\CI$ can be definably embedded in one of four weakly o-minimal structures: $\la K;<\ra$, $\la k;<\ra$, $\la K/O;<\ra$ or $\la \Gamma;<\ra$ (the value group).
By analysing definable functions in $K/O$ and showing that they are locally ``affine'' with respect to the additive structure of $K/O$, we eliminate the possibility that $I$ could be embedded into $K/O$. As for $\Gamma$, we consider two cases:
 If $T$ is power-bounded then the o-minimal $K$-induced structure on $\Gamma$ cannot interpret a field (see \cite{vdDriesTconv}), and in fact we show that $I$ cannot be embedded in $\Gamma$. In the exponential case we show that $\la \Gamma,+\ra $ is definably isomorphic to $\la K/O,+\ra$ and therefore, by the above, $I$ cannot be embedded into $\Gamma$. Thus, regardless of whether $T$ is power bounded or exponential, we are left with the first two possibilities, of $K$ and $k$. Using the notion of infinitesimals we prove
in these two cases that the field is definably isomorphic to a definable field in $K_{conv}$ or in $k$. The final result follows from the work on definable fields in o-minimal structures (\cite{OtPePi}).

We remark that given the o-minimal expansion of $K$ in a signature $\CL$, as in the statement of Theorem \ref{main}, it  suffices to prove the result for some $\la K;O\ra \prec \la K';O\ra $. Therefore, throughout, \textbf{we tacitly assume that $\la K;O\ra $ is $(|\CL|+\aleph_0)^+$-saturated} as will be all structures considered below.

\vspace{.3cm}

As a corollary to our main theorem, and using the work of Hempel and Palacin \cite{Hempel-Palacin}, we obtain a theorem about definable division rings:
\begin{cor} Let $K$ be a real closed valued field, or more generally a model of $T_{conv}$, where $T$ is an o-minimal
power bounded
expansion of a real closed field.   If $D$ is a division ring interpretable in $K$ then $D$ is definably isomorphic to either $K$, $K(\sqrt{-1})$, or the quaternions over $K$, or to  $k$,  $k(\sqrt{-1}))$, or the quaternions over $k$.
\end{cor}
\proof By \cite[Theorem 2.9]{Hempel-Palacin}, $D$ is a finite extension of its center,  the field $\CF$. By the theorem above, $\CF$ is definably isomorphic to a field which is definable in $\la K;+,\cdot\ra$, or  in $\la k;+,\cdot\ra $. Thus $D$, as a finite extension, is definably isomorphic to a division ring definable in these. The result follows from \cite[Theorem 1.1]{OtPePi}
on division rings in o-minimal structures.\qed\\

\begin{remark*}
\begin{enumerate}
\item Weakly o-minimal structures are dp-minimal (namely, have dp-rank $1$), thus if $\CF$ is an interpretable field in a real closed valued field (or in a model of $T_{conv}$) then $\dpr(\CF)<\omega $. It is therefore possible that the algebraic classification of $\CF$, into a real or an algebraically closed field of characteristic $0$, follows from the seminal work of Johnson on fields of finite dp-rank (see for example \cite{johnsondpfiniteVI}). Note however that our theorem gives the additional connection between $\CF$ and the fields $K$ and $k$.

\item A weakly o-minimal theory is distal \cite[Chapter 9]{SiBook}, as is  $T^{eq}$ (see Exercise 9.12 there). Thus an interpretable field in $T_{conv}$ is definable in a distal structure. As was pointed out by Chernikov, such fields cannot have characteristic $p>0$ due to the combinatorial regularity in \cite{CherStarch}. So our field $\CF$ must have characteristic zero.

However, our arguments do not make use of either of the above algebraic facts.

\item As the two comments above suggest, there are various model theoretic frameworks which fit our setting. The field $\CF$ has finite dp-rank. It is also definable in a distal structure. Finally, the structure $\la K;O\ra $, as well as the induced structures on $K/O$, $k$ and $\Gamma$ are al weakly o-minimal expansions of a groups thus they are all uniform dp-minimal structures in the sense of Simon and Walsberg,  \cite{SimWal}. While no direct application of their paper appears in this final version we found the results there to be very helpful during our work on this project.
\item We expect that a similar direct proof could be given for the theorem of Hrushovski and Rideau-Kikuchi, on interpretable fields in algebraically closed valued fields.
    
    \item An earlier version of the article claimed the result, wrongly, for $T$ the theory of an arbitrary o-minimal expansion of a real closed field. We thank Yatir Halevi for noticing the gap.
    \end{enumerate}
\end{remark*}

%\section{Preliminaries}
%
%\subsection{Weakly o-minimal structures}
%
%$\bullet$ Dimension of sets and of tuples. Generic points.
%
%\vspace{.2cm}
%
%$\bullet$ sub-additivity.
%
%\vspace{.2cm}
%
%$\bullet$ Section 5 to insert here?

\subsection{Acknowledgments} 
%\footnote{An earlier version of the article claimed the result wrongly for arbitrary o-minimal expansions of $K$}

 A second article, in which similar methods are applied to analyze interpretable fields in P-minimal fields, was written jointly with Yatir Halevi, see \cite{Pminimal}, and we plan to combine both articles into one joint paper.

We thank the model theory group in D{\"u}ssseldorf for their careful reading of the current article. Their helpful comments will be incorporated into the  final version of the paper.
 
\section{Dimension in weakly o-minimal theories}	
	%subset with its natural ordering.
	
{\em Throughout this section $\mathcal M$ denotes a model of a weakly o-minimal theory.
	.}

\vspace{.2cm}

%{\bf Let's write this more expanded and include what we wrote above}

	%We assume that $I$ is definable over $\0$.
We first collect some useful facts concerning dimension in weakly o-minimal structures.
The set $M$ is equipped with the order topology and $M^n$ with the product topology. By {\em an open box} in $M^n$ we mean a cartesian product of convex subsets of $M$ (since $M$ is weakly o-minimal the end points might not be in $M$).

We recall that if $\mathcal M$ is a weakly o-minimal structure and $S\sub M^r$ is a definable set then  $\dim(S)=l$ if $l$ is the maximal natural number  such that the projection of $S$ on $l$ of the coordinates has non empty interior.

 We remind that definable functions in weakly o-minimal structures may be locally constant without being piecewise continuous, so that $\dcl$ need not, in general, satisfy the exchange principle. Though $\dcl$ satisfies exchange in the main sorts of $T_{conv}$ (because it coincides with $\dcl_{\CL}$), it is not the case for  $K/O$. \\	

 The following  can be inferred readily from \cite[Theorem 3.8]{Simdpr} and the subsequent comment.\\
\begin{fact} \label{dp=dim}If  $Y\sub M^n$ is definable then $\dpr(Y)=\dim(Y)$.\end{fact}

	For $a$ a tuple in $M^n$ and $A\sub M$ we let $\dim(a/A)$
	be the minimal dimension of an $A$-definable subset of $
	M^n$ containing $a$. Given $X\sub M^n$ definable over $A$, we say that $a$ is {\em generic in $X$ over $A$} if $a\in X$ and $\dim(a/A)=\dim(X)$.
 Our standing saturation assumption assures that generics over small parameter sets always exist. \\
%	It follows from the above (and see also \cite[Proposition 2.4]{SimWal}) that this equals the $\acl$-dimension of $a$ over $A$, namely the maximal size of a subtuple $\acl$-independent over $A$.  By the above fact, this equals the minimal $\dpr$ of an $A$-definable set containing $a$.
%	
	\begin{cor} \label{dp=dim-types}If $a\in M^n$ and $A\sub M$ then $\dpr(a/A)=\dim(a/A)$.\end{cor}
	\proof By \cite[Corollary 3.5]{Simdpr} in dp-minimal structures, $\dpr()$ is local, namely $\dpr(a/A)$ is the minimal dp-rank
	of an $A$-definable set containing $a$. By definition, the same is true for $\dim(a/A)$, so we can use Fact \ref{dp=dim}.\qed \\
	
	By the sub-additivity of the dp-rank, \cite{KOUadd}, we have:
\begin{fact}
	$\dim(ab/A)\le \dim(a/A)+\dim(b/Aa).$
\end{fact}
	
We shall use several times the fact that a definable subset of $M^n$ in weakly o-minimal theories has a finite decomposition into definable cells, each homeomorphic, via an appropriate projection to an open set in some $M^k$ (see \cite[Theorem 4.11]{MacMaSt}). Also, we shall use the fact, \cite[Theorem 4.7]{MacMaSt}, that dimension of definable sets is preserved under definable bijection.
	
	 From now on we only use the weakly o-minimal dimension (and not dp-rank), providing -- for the sake of completeness -- self contained proofs of the properties we need.
	
	\begin{lemma}\label{external} Assume that $X\sub M^n$ is definable over $A$ and that there is an externally defined set $W$ (i.e. $W$ is definable in an elementary extension $\CN$) such that $X(M)\sub W$. If $\dim(W)\leq s$ then $\dim(X)\leq s$.
	\end{lemma}
	\proof
	%This follows from Assaf's argument for $dp$-rank and Corollary \ref{dp=dim}.
%	This follows readily from the characterisation of $dp$-rank using $ict$-patterns (see \cite[Proposition 4.21]{SiBook}) since the latter does not depend on parameters. I.e., an ict-pattern for $X$ in $M$ is automatically an ict-pattern for $W(M)$, which is automatically an ict-pattern for $W$. So $\kappa_{ict}(X)\le \kappa_{ict}(W)$. We give a self-contained direct argument:
	
	By weak o-minimality, there is a cell decomposition $C_1,\dots, C_l$ for $W$ and projections $\pi_i: C_i\to M^{t_i}$ such that $\pi_i(C_i)$ is open and $\pi_i|C_i$ is a homeomorphism. So $\dim(C_i)=t_i$. It follows that we can find $N$-definable sets $W_1,\ldots, W_r\subseteq W$ such that every $W_i$ has a finite-to-one projection onto $s$ of the coordinates. Without loss of generality, $W$ is one of these $W_i$ and hence there is a finite-to-one projection $\pi:W\to N^s$. Since $X(M)\sub W$ the restriction of $\pi$ to $X(M)$ is also finite-to-one. Hence $\dim(X)\leq s$.\qed \\

	We also need the following:
	\begin{lemma} \label{dp} Let $(J,\fleq)$ be a definable infinite linearly ordered set without a maximum element. Let $\{X_b:b\in J\}$ be a definable family of subsets of $M^n$ such that for $b_1\fleq  b_2$, $X_{b_1}\sub X_{b_2}$. Assume that for every $b\in J$ we have $\dim(X_b)\leq m$. Let $U=\bigcup_{b\in I} X_b$. Then $\dim(U)\leq m$.\end{lemma}
	\proof  Let $b^*\in J(N)$ be an element in an elementary extension $\mathcal N$, such that $b^*$ is greater than all elements of $J(M)$ (since $J$ has no maximum such a $b^*$ exist). Since dimension is definable in parameters, $\dim(X_{b^*})\leq m$.  For every $b\in M$, $X_b(M)\sub X_{b^*}$ therefore $U(M)\sub X_{b^*}$.  It follows from Lemma \ref{external} that $\dim(U)\leq m$.\qed \\
	
\begin{lemma}\label{dp2} Let $c_1,\ldots, c_m$ be each a tuple of elements from $M$, of possibly different lengths, and let $C\sub M$ be an $A$-definable convex set.

Given any initial segment $I\sub C$ (possibly defined over additional parameters), there exists $b\in I$ such that
for every $i=1,\ldots, m$, $\dim(bc_i/A)=1+\dim(c_i/A)$.
\end{lemma}
\begin{proof} For each $i=1,\ldots, m$ we fix $W_i$, definable over $A$, such that $\dim(c_i/A)=\dim W_i$.
Consider the type $p(x)$:
$$\{x\in I\}\cup \bigcup_{i=1}^m\{\neg X(x,c_i): \dim(X)\leq \dim(c_i/A) \mbox{ and $X\sub M\times W_i$ definable over $A$}\}.$$
Since $C$ is definable over $A$, we may take above only those sets $X$ such that for every $w\in W_1$, $X(M,w)$ is a nonempty subset of $C$.

If $p$ is consistent then its realization will be the desired $b$. So, we assume towards contradiction that $p$ is inconsistent.

It follows that there are $X_1,\ldots, X_r$ as above, all $A$-definable, and for each $j=1,\ldots, r$, there is $i(j)\in\{1,\ldots,m\}$, such that
$$\forall x\,\, \left ( x\in I\rightarrow \bigvee_{j=1}^r X_i(x,c_{i(j)})\right ).$$
In addition, for each $j=1,\ldots, r$, $\dim(X_j)\leq \dim(W_{i(j)}).$

By weak o-minimality, each set $X_j(M,c_{i(j)})$ is a finite union of convex subsets of $C$, and since $I$ is a non-empty initial segment of $C$, the above implies that for some fixed $j$, the set $X_j(M,c_{i(j)})$ contains an initial segment of $C$. For simplicity, we write $X$ for $X_j$ and assume that $i(j)=1$. Namely, we assume that $X(M,c_1)$ contains an initial segment of $C$. This is an $A$-definable property of $c_1$, therefore we may assume, after possibly shrinking $W_1$, that for every $w\in W_1$, the set $X(M,w)$ contains an initial segment of $C$. Moreover, by choosing in each $X(M,w)$ the first convex component,  we may assume that for every $w\in W_1$, the  set $X(M,w)$ contains a single convex set, which is an initial segment of $C$.

For $b\in C$, let $X_b=\{w\in W_1:\la b,w\ra\in X\}$. Notice that $\bigcup_{b\in C} X_b=W_1$.
 Our assumptions imply that for $b_1\leq b_2\in C$, we have $X_{b_2}\sub X_{b_1}$. It now follows from lemma \ref{dp} (with the ordering $<$ reversed) that for some $b_0\in C$ we have $\dim(X_{b_0})=\dim(W_1)$.

 Also, our assumptions imply that for every $b<b_0$ in $C$ and $w\in X_{b_0}$ we have $(\la b,w\ra \in X$, namely $C^{<b_0}\times X_{b_0}\sub X$. Thus, $\dim(X)\geq 1+\dim X_{b_0}> \dim W_1$. But by our assumption,  $\dim(X)\le \dim(c_1/A)=\dim(W_1)$. This is a contradiction and the lemma is proved.\end{proof}

As a corollary we obtain:
\begin{cor}\label{dp3} Assume that $\CM$ has definable Skolem funcions, and let $X\sub M^n$ be an $A$-definable set. Let $c_1, \ldots, c_m$ be tuples from $M$ (of possibly different lengths).
Then there exists $b\in X$ such that for each $i=1,\ldots, m$, $$\dim(bc_i/A)=\dim(X)+\dim(c_i/A).$$
\end{cor}
\begin{proof}  By cell decomposition we may assume that $X$ is an open subset of $M^n$. We use induction on $n$.

The case $n=1$ follows from Lemma \ref{dp2}. For a general $n$, note first that $X$ contains an $A$-definable open box $I_1\times\cdots\times I_n$. Indeed, since $X$ is open it contains an open box and by the existence of Skolem functions we can find such a box defined over $A$. Applying induction, we can find $a'=\la a_2,\ldots, a_n\ra\in I_2\times\cdots\times I_n$ such that
for each $i=1,\ldots, m$, $\dim(a'c_i/A)=(n-1)+\dim(c_i/A)$.

We now let $c_i'=\la a',c_i\ra$, for $i=1,\ldots, m$, and using again the case $n=1$, find $a_1\in I_1$ such that , for all $i$
$\dim(a_1c_i'/A)=n+\dim(c_i/A)$. The tuple $\la a_1a'\ra $ is the desired $b$.\end{proof}

In the above existence of Skolem functions somewhat simplifies the proof, but is not, in fact, needed. \\

\subsection{Dimension and domination in quotients by convex equivalence relations}
We establish here several properties of quotients of one-dimensional sets by convex equivalence relations.
 Recall:
\begin{definition}
	An equivalence relation $E$ on a linearly ordered set $(J,<)$ is convex if all its equivalence classes are.
\end{definition}

If $\mathcal M$ is weakly o-minimal, $E$ a definable convex equivalence relation with infinitely many classes then $M/E$ is linearly ordered and itself weakly o-minimal. If in addition $Th(\CM)$ is weakly o-minimal then so is $Th(\CM/E)$. It follows, in particular (see next section), that if $\la K;O\ra\models T_{conv}$ then the theories of $K$, $K/O$, $\Gamma$ and $k$ are weakly o-minimal.

Let $\CM$ be a weakly o-minimal structure, $E$ a definable convex equivalence relation on $M$, and let $\pi:M\to M/E$ be the quotient map. For simplicity we also denote by $E$ the equivalence relation on $M^n$ given by $E^n$ and let $\pi:M^n\to (M/E)^n$ be the associated quotient map.
We prove here a domination result for such quotients.

\begin{lemma} \label{opaque} Let $I_1,\ldots, I_n\sub M$ be open convex sets such that each $I_j/E$ is infinite, and $X\sub B=I_1\times\cdots\times I_n$. If $\pi(X)=\pi(B)$ then $\dim(\pi(B\setminus X))<n$.\end{lemma}
\proof

%{\bf Not finished}

Let us say that a definable set $X$ {\em crosses} an $E$-class $[a]$ if both $X$ and $X^c$ intersect $[a]$.

%We first observe that for $n=1$ the lemma follows from weak o-minimality.  Indeed, if $X\sub I$ is definable then, since $X$ consists of at most $r$-convex sets and each $E$-class is convex, then $X$ crosses at most $2r$-many classes. Thus, since $X$ intersects all $E$-classes, $\pi(I_1\setminus X))$ is finite.
%
%So our assumption implies that $Y:=B\setminus X$ crosses every class $c\in \pi(Y)$. For simplicity of notation denote $y\in B$ as $(a,b)$ with $a\in B':=\prod_{j=1}^{n-1}I_j$ and $b\in I_n$. By weak o-minimality of the theory of $M$ there is a uniform bound, $l$, on the number of convex components of $Y_a$ as $a$ ranges over $B'$. Therefore, we may assume that, in fact $Y_a$ is convex for all $a$.
%
%Since $E$ is convex if $\pi(b_1)<\pi(b_2)\in I_n/E$ then for any $a\in B'$ we observe that $B_i:=[a,b_i]$ are disjoint open boxes whose projections onto the first $(n-1)$-coordinates are the same. Assume that $\pi(b_1)<\pi(b_2)$, that $[a,b_1]\neq [a,b_2]$ and that $\pi(a,b_i)\in \pi(Y)$ for $i=1,2$ and some $a$. The assumption that $Y_c$ is convex for all $c\in B'$ implies that there cannot be $b_3$ with $\pi(b_1)<\pi(b_3)<\pi(b_2)$ because if such a $b_3$ existed then our convexity assumption would imply that $[a,b_3]\subseteq Y$, contradicting the observation that any class in $Y$ crosses $Y$. This shows that $\pi(Y)$ has empty interior, and therefore cannot be $n$-dimensional.
%
%\qed \\

We proceed by induction on $n$. The result for $n=1$ follows from weak o-minimality of $K/E$: Indeed, if $X\sub I$ is definable then, since $X$ consists of at most $r$-convex sets and each $E$-class is convex, then $X$ crosses at most $2r$-many classes. Thus, since $X$ intersects all $E$-classes, $\pi(I_1\setminus X))$ is finite.

Assume now that $X\sub M^n$ and $\pi(X)=\pi(B)$. let $Y=B\setminus X$,  and assume towards a contradiction that $\dim(\pi(Y))=n$. This implies that $\pi(Y)$ has nonempty interior, and therefore contains a box. The pre-image of such a box is itself a box (since $E$ is convex, and we may assume that its domain is an open box).  Thus there exists an open box $B_1\sub Y$ such that $\pi(B_1)$ is an open subset of $\pi(X)$. By our assumption, we also have $\pi(B_1\cap X)=\pi(B_1)=\pi(B_1\cap Y)$.

Thus we may assume, by replacing $B$ with $B_1$,  that $\pi(X)=\pi(Y)=\pi(B)$.  To simplify notation let $B'=I_1\times \cdots \times I_{n-1}$.
We define
$$X^*=\{\la a,b\ra\in B'\times I_n: X_a \mbox{ crosses }[b]\}.$$

By the case $n=1$, for each $a\in B'$, $X_a$ crosses at most finitely many $E$-classes in $I_n$.
It follows that $\pi(X^*)\sub \pi(B)$ projects finite-to-one into $\pi(B')$, and thus $\dim(\pi(X^*))<n$.
Hence, there exists an open box $R\sub \pi(B)$ which is disjoint from $\pi(X^*)$. By replacing $B$ with $\pi^{-1}(R)$ we may assume that $X^*=\0$. Namely, for every $\la a,b\ra\in B'\times I_n$, either $X_a$ contains $[b]$ or it is disjoint from $[b]$.

For every $b\in I_n$, we let $X^b=\{a\in B':\la a,b\ra\in X\}$. We claim that for each $b\in I_n$, $\pi(X^b)=\pi(B')$. Indeed, if $a\in B'$ then there is $\la a',b'\ra\in X$ such that $\la a',b'\ra E\la a,b\ra$ (since $\pi(X)=\pi(B)$). By our assumption, $X_{a'}$ does not cross $[b']$, therefore $[b']$ is contained in $X_{a'}$. We have $b E b'$, hence $\la a',b\ra \in X$. It follows that $\pi(a)=\pi(a')\in \pi(X)$. Our induction assumption thus implies that $\dim(B'\setminus X^b)<n-1$.

We also claim that for $b_1E b_2$, we have $X^{b_1}=X^{b_2}$ (and hence also $B'\setminus X^{b_1}=B'\setminus X^{b_2}$). Indeed, if $a\in X^{b_1}$ then $X_a\cap [b_1]\neq \0$ and therefore $[b_1]\sub X_a$, so $b_2\in X_a$, namely $a\in X^{b_2}$. The opposite inclusion follows.

Thus, for every  $b\in I_n$, we have
$\pi(B\setminus X)^{\pi(b)}=\pi(B'\setminus X^b)$, which has dimension $<n-1$. It follows that  for every $c\in \pi(I_n)$, $\dim(\pi(B\setminus X)^c)<n-1$. By the sub-additivity of dimension, $\dim(\pi(B\setminus X))<n$, contradicting our assumption.\qed \\

The following could be viewed as domination of types in $M^n$ by  generic types in $M^n/E$.

\begin{proposition} \label{opaque1}In the above setting, assume that $p\vdash (M/E)^n$ is a complete generic type over $A\sub M$. Then $q=\pi^{-1}(p)$ is a complete type over $A$.
\end{proposition}
\proof If $q$ is not complete then there is an $A$-definable set $X\sub M^n$ such that $q\vdash \pi(X)$ and $q\vdash \pi(X^c)$. But then $\dim(\pi(X))=\dim(\pi(X^c))=n$, contradicting Lemma \ref{opaque}.\qed

\section{The theory $T_{conv}$}
\label{sec T-convex}
Let $T$ be an o-minimal expansion of a real closed field in a signature $\CL$. The theory $T_{conv}$ was introduced by Lowenberg and v.d. Dries in \cite{vdDrLew}:
Given $K\models T$, {\em a $T$-convex subring of $K$} is a subring $O\subseteq K$ with the property that $f(O)\subseteq O$ for any $\0$-definable continuous function $f:K\to K$. Since $O$ is convex it is a valuation ring, its maximal ideal is denoted $\mu$, the value group is $\Gamma$ and the residue field is $k$. The theory $T_{conv}$, in the language $\CL_{conv}:=\CL\cup \{O\}$, is the extension of $T$ by the axioms saying that $O$ is a proper valuation ring which is  $T$-convex.
Here are some results from \cite{vdDrLew} and \cite{vdDriesTconv}:

\begin{fact}\label{conv-L}
	\begin{enumerate}
		
		\item The theory  $T_{conv}$ has quantifier elimination relative to $T$, it is complete and weakly o-minimal. (\cite[3.10, 3.13,3.14]{vdDrLew}.\\

		\noindent Let $T_{conv,c}$ be the extension of $T_{conv}$ by the formula $c\notin O$, for a new constant $c$. Then:
		
		\item   The theory $T_{conv,c}$ has definable Skolem functions (\cite[Remark 2.4]{vdDriesTconv}).
		
		\item In models of $T_{conv,c}$ the definable closure in $\CL$ and in $\CL_{conv}$ are the same (see \cite[Lemma 2.6]{vdDriesTconv}). Moreover, every $\CL_{conv}$-definable function coincides with an $\CL$-definable function around generic points of the domain.
		
	\end{enumerate}
\end{fact}

\emph{From now on, we add a constant symbol $c$ interpreted as a positive element outside $O$,  and still denote the new language $\CL_{conv}$.
} Throughout the text we will be using the above results without further reference.\\

\begin{defnn} Let $\CM$ be an o-minimal expansion of  real closed field $K$. \emph{A power function on $K$} is a definable endomorphism of the ordered mutiplicative group $K^{>0}$.

$\CM$ is called {\em power bounded} (generalizing ``polynomially bounded'') if  every definable function of one variable is eventually bounded by some power function. An o-minimal theory $T$ is called {\em power bounded} if every model of $T$ is power bounded.
\end{defnn}

By Miller's \cite{MilStar}, $\CM$ is power bounded if and only if it is not exponential, namely one cannot define a (necessarily monotone) isomorphism of $(K,+)$ and $(K^{>0},\cdot)$.

The structure of the residue field and the value group of $T_{conv}$ is described by the work of v.d. Dries :

\begin{fact}[Theorem A, Theorem B]\cite{vdDriesTconv}]\label{factt}
	Let $T$ be the theory of an o-minimal expansion of a real closed field $K$ and $O$ a $T$-convex subring of $R$. Then:
	\begin{enumerate}
		\item The residue field with its induced structure can be given a structure of a $T$-model, and it is stably embedded as such. In particular it is o-minimal, so by \cite{OtPePi}, any field interpretable in the residue field is definably isomorphic to the residue field itself or to its algebraic closure.
		\item If, in addition, $T$ is power bounded, then the value group is, up to a change of signature, an ordered vector space over the field of exponents of $T$ and is stably embedded as such.
	\end{enumerate}
\end{fact}

We shall also use the following result from J. Tyne's PhD thesis \cite[p.94]{Tyne}:
\begin{fact}\label{power bounded} Let $T$ be power bounded and o-minimal and $K{conv}\models T_{conv}$. Then every $\CL_{conv}$-definable subset of $K$ ia already definable in the valued field, namely it is a boolean combination of balls and intervals.
\end{fact}

To sum up the last facts in few words, if we denote $k_{ind}$ the induced structure on $k$ (in the signature $\CL_{ind}$) then $k_{ind}$ is o-minimal, and up to a change of signature elementarily equivalent to $R$. In case $T$ is power bounded, the value group, too, is an o-minimal vector space with no additional structure and the $1$-variable definable subsets of $K$ are the same as in the underlying real closed valued field.

\textbf {\em From now on, we denote the induced structure on $k$ by $k_{ind}$ and its language by $\CL_{ind}$.}\\

We finally note:
\begin{lemma}\label{exp}
Let $T$ be an exponential o-minimal expansion of a real closed field, $T_{conv}$ its expansion by a $T$-convex valuation ring. If $K\models T$ then $\la K/O;<,+\ra $ is in definably isomorphic with $\la \Gamma;<,+\ra$.
\end{lemma}
\begin{proof} We first note that $\exp(O)=O_{>0}\setminus \mu$.  Indeed, the right-to-left inclusion follows from the fact that $O$ is $T$-convex, and $exp(x)>x$. For the converse, assume for contradiction that  $a=\exp(b)\in \mu$ for some $b\in O$. Then $a^{-1}=exp(-b)\notin O$, contradicting the fact that $O$ is $T$-convex.

Since $O\setminus \mu=O^{\times}$ it follows that $exp$ induces a definable isomorphism between $\la K/O;<, +\ra$ and $\la K_{>0}/O_{>0}^\times;<,+\ra $, which is isomorphic to $\la \Gamma;<,+\ra$.
\end{proof}
 For example, in this case the residue field is internal to the value group, and the latter's theory (relative to $T$) is not known (see \cite[Clause 6.5]{vdDriesTconv}).

\vspace{.2cm}

It will be convenient to define in $\CL_{conv}$ the relation $v(x)>v(y)$ if $x/y\in O\land y/x
\notin O$. Since $O=\{x: v(x)\ge v(1)\}$, quantifier elimination for $T_{conv}$ in $\CL_{conv}$ is equivalent to quantifier elimination for $\CL$ expanded by the the above binary predicate, and we will use the two languages interchangeably.

The following could probably be read off \cite{vdDriesTconv}, but we need the  explicit formulation below. It is a consequence of quantifier elimination relative to $T$:

\begin{lemma}\label{TconvQE}
	Let $S\subseteq K^n$ be $\CL_{conv}$-definable, $a\in S$ generic over the parameters defining $S$ (in the weakly o-minimal sense). Then there exists an open box $B\ni a$ (defined possibly over  new parameters) such that $B\cap S$ is $\CL$-definable.
\end{lemma}
\begin{proof} We first note that the associated valuation induces the same topology on $K$ as the order topology.
	Quantifier elimination implies that any definable set is of the form
	\[
	\bigcup_i \left ( \bar T_{i}(\bar x)\cap \left ( \bigcap_j (t_{i,j}(\bar x)\in O)^{\pm 1} \right )\right )
	\]
	where $\bar T_i$ are $\CL$-definable sets, $t_{i,j}(\bar x)$ are $\CL$-terms and $(x\in O)^{-1}$ means $x\notin O$. It suffices to verify the lemma for sets of the form
	\[
	\bar T(\bar x)\cap \left ( \bigcap_j (t_{j}(\bar x)\in O)^{\pm 1} \right )
	\]

	The set $O$ is clopen in $K$ and since $a$ is generic with respect to the o-minimal structure, each $t_{i,j}$ defines a continuous function at $a$. It follows that $a$ is an interior point of both  the preimage of $O$ and its complement under each $t_j$. Thus there is an open box $B\sub K^n$ containing $a$ such that $B$ is contained in the set $ \bigcap_j (t_{j}(\bar x)\in O)^{\pm 1}$. We have $T\cap B= S\cap B$, so it is $\CL$-definable.
\end{proof}

\section{Definable quotients of $K$}

In this section we classify, locally, definable quotients of $K$ itself. The first lemma will eventually allow us to reduce much of the work from quotients of the form $K^n/E$ to quotients of the form $(K/E)^n$.

\begin{lemma}\label{1-dim}
	Let $\mathcal M$ be a weakly o-minimal structure with definable Skolem functions. Let $S=X/E$, for some definable $X\subseteq K^n$ and $E$ a definable equivalence relation. If $S$ is infinite then there exists an infinite $S'\subseteq S$ (definable over additional parameters) such that $S'$ is in definable bijection with $J/E'$ for some interval $J\subseteq M$ and $E'$ a definable equivalence relation on $M$.
\end{lemma}
\begin{proof}
	We use induction on $n$, where the case $n=1$ is trivial. Let $W:=\pi(X)$ be the projection of $X$ onto the last $n-1$ coordinates (assuming, without loss of generality, this projection has infinite image). If for some $w\in W$ the fibre $\pi^{-1}(w)$ meets infinitely many $E$-classes then take $S'=X^w=\{x\in M:\la x,w\ra\in X\}$ and $E'$ the relation on $S'$ obtained from $\pi^{-1}(w)$.

 So we may assume that $|\pi^{-1}(w)/E|<\infty$ for all $w\in\pi(X)$. Since $\mathcal M$ has  definable Skolem function we can find, uniformly for each $w\in W$ a finite set of representatives $Y(w)$ of $\pi^{-1}(w)/E$. More precisely, there exists a definable set $Z\sub K^n$ such that $\pi(Z)=\pi(X)$
and for each $w\in \pi(X)$, $Z\cap \pi^{-1}(w)$ has exactly one representative for each $E$-class intersecting $\pi^{-1}(w)$.

 %Let $X':=\{\la w,Y(w)\ra : w\in W\}$. Note that $X'/E=X/E$ and $\dim(X')<n$.
	
	 By weak o-minimality, there exists a cell decomposition $Z_1,\dots, Z_l$ of $Z$ such that each $Y_i$ is definably homeomorphic to an open cell $C_i\subseteq M^{\dim(Z_i)}$. Assume, without loss, that $|Z_1/E|=\infty$. If $\pi:Z_1\to C_1$ is the definable homeomorphism let $E_1:=\pi(E|Z_1)$ and we conclude by induction.
	 %by induction there is some $I\subseteq M$ and some $E'$ such that $C_1/E_1=Y_1/E$ is definably isomorphic to $I/E'$, as required.
\end{proof}

We now fix an infinite field $\la \CF;+,\cdot\ra $ interpretable in $K$.  We fix an interval $J\subseteq K$ and a definable equivalence relation $E$ on $J$ as provided by the above lemma. As we show next, after possibly shrinking $J$, the set $J/E$ can be definably  embedded into one of four weakly o-minimal structures.

\begin{theorem}\label{K/E}
	Let $K\models T_{conv}$, where $T$ is power bounded, and let  $E$ be a definable equivalence relation on $J\sub K$ with infinitely many classes. Then there exists a definable ordering on $J/E$ such that $\la J/E,<\ra$, with the induced structure,   is weakly o-minimal, and there exists an $<$-interval $I\subseteq J/E$ which is in definable bijection with an interval in $K$, in $\Gamma$, in $k$ or in $K/O$.
\end{theorem}
\proof
By weak o-minimality each $E$-equivalence class is a finite union of convex sets. So by replacing $E$ with the equivalence relation $E'$, choosing the first component in each $E$-class, we get a convex equivalence relation with $D/E'=J/E$ for some definable $D\subseteq J$. So we assume that $E=E'$. Thus $J/E$ is linearly ordered and weakly o-minimal.

If $E$ is the identity on an infinite set we can find a convex set $I\sub J$ where this is true,  and then $I/E\hookrightarrow K$, so  we are done. So we assume that this is not the case. It follows that for all but finitely many classes, $\sup([x])$ is not an element of $K$ (for otherwise the set $\exists y (x=\sup[y])$ is an infinite discrete set in $K$). For the same reason $\inf([x])$ is not an element of $K$ for all but finitely many of the $E$-classes. By shrinking $J$ we may assume that  for all $x\in J$, the class $[x]$ is a bounded convex set without supremum or infimum.

For a valuational ball $B$ (either open or closed) write $x<B$ if $x<y$ for all $y\in B$. We write $x\le B$ if $x<B\lor x\in B$.  Recall (see, e.g., \cite{Mellor}) that a \emph{disc cut} is a set of the form $x \square B$ where $\square \in \{<,\le \}$ and $B$ is a definable ball (either open or closed) or a point in $K$. By Tyne's work, Fact \ref{power bounded}, every $\CL_{conv}$-definable subset of $K$ is a finite boolean combination of disc cuts. Since the $E$-classes are convex it follows from our assumptions that they are intersections of two disc cuts where $B$ is actually a ball (and not a point).

Let us call a set of the form $B_1\square_1 x \square_2 B_2$, where $\square_i\in \{<,\le\}$,  \emph{a ball-interval}. The \emph{form} of  a ball interval is the data specifying  which of the balls $B_i$ is open, and which of the inequalities is weak. \\

Below we use valuational radius for balls, namely for $\gamma\in \Gamma$,  an open ball of radius $\gamma$ is a set of the form $\{x\in K:v(x-x_0)>\gamma\}$ and a closed ball is of the form $\{x\in K:v(x-x_0)\geq \gamma\}$.

\vspace{.2cm}

\noindent\textbf{Claim}. If $C$ is a ball interval $B_1\square_1 x \square_2 B_2$  then $B_1$ and $B_2$ are uniquely determined by $C$. More explicitly, $B_1$ and $B_2$ are definable over a code for $C$.

\proof
	We deal first with the case where $\square_1$ is $\leq$. If $B_1$ is a ball of radius $\gamma$ (either closed or open) then $$\gamma=\inf\{\delta: (\forall x\in C \exists y \in C)(y<x \land B_\delta(y)\subseteq C)\}.$$
	And $B_1=B_\gamma(y)$ for any small enough $y\in C$. We can recognise whether $B_1$ is open or closed by verifying whether $\gamma$ is attained as a minimum in the above formula or not. $B_2$ is defined similarly when $\square_2$ is $\leq$.

 The case of strong inequalities follows from the previous case by replacing $I$ with its complement.
\qed$_{\mbox{claim}}$.\\

The claim allows us to definably assign to each class $[x]$ the left and righ ``end-balls''  of the ball-interval,  $B_1[x]$ and $B_2[x]$, respectively. There are finitely many possible forms of ball-intervals, hence we may further assume that all equivalence classes $[x]\sub J$ have  the same form. Note that if $[x]\neq [y]$ then $B_1[x]\neq B_1[y]$, for  otherwise $[x]\cap [y]\neq \0$. So the function $[x]\mapsto B_1[x]$ is a definable injection into the family of open (closed) balls. If for some $\gamma\in \Gamma$ the fibre of the function sending $[x]$ to the valuational radius of $B_1[x]$ is infinite then, depending on whether $B_1[x]$ is open or closed, we get an injection of some interval in $J/E$ into $K/O$ or $k$.

 If all fibres of this function are finite we get a finite-to-one map from an interval in $J/E$ into $\Gamma$. The linear order on $J/E$ implies the existence of an injection from an interval in $J/E$ into $\Gamma$, finishing the proof of Theorem \ref{K/E}.\qed

\section{Definable functions in $K/O$}

 Combining Lemma \ref{1-dim} with Theorem \ref{K/E} we conclude that there exists an infinite definable $I$ in our interpretable field $\CF$ such that $I$ is in definable bijection with an interval in one of the four  weakly o-minimal structures  $K$, $\Gamma$, $k$ or $K/O$.
 While the definable sets in $K$, $\Gamma$ and $k$ are well understood, the situation is less clear in $K/O$. Our goal in this section is to show that definable functions in $K/O$ are locally affine with respect to the additive structure of $K/O$. The ``polynomial boundedness'' assumption is not used in this section.\\

We start with the following general lemma.

\begin{lemma}\label{convex}
Assume that $p\vdash K^n$ is a partial  $\mathcal L$-type and that $H:K^n\to K$ is a partial $\CL$-definable function such that $p\vdash \dom(H)$.

Assume also that $C\sub K$ is an $\CL_{conv}$-definable convex set which is bounded above (or below) and for every $\alpha\models p$, $H(\al)\in C$. Then there exists $c\in C$ such that for every  $\alpha\models p$ we have $H(\al)<c$ (or $H(\al)>c$).
\end{lemma}
\proof We are using the saturation of $\CM_{conv}$, and assume that $C$ is bounded above.
For every $\CL$-definable set $X\sub K^n$, for which $p\vdash X$, let $s(X)$ be the supremum of $H(W)$, which exists by o-minimality and the boundedness of $C$. We claim that for some such $X$ we have $s(X)\in C$ (and then we may take $c=s(X)$). Indeed, if not then for every $X$ in $p$ there exists $\alpha\in X$ such that $s(\alpha)>C$.  But then, by saturation we may find $\alpha\models p$ such that $H(\alpha)>C$, contradicting our assumption.\qed

\begin{definition} A  subset of $K$ is called {\em long} if it contains infinitely many cosets of $O$. A subset of $K^n$ is {\em long} if it contains  a cartesian $n$-product of long subsets of $K$. 	A type  $p\vdash K^n$ is long if every set in $p$ is long.
\end{definition}

\begin{definition} We say that a (partial) function $F:K^n\to K$ {\em descends to $K/O$} if whenever $a-b\in O^n$ also $F(a)-F(b)\in O$ .\end{definition}

\begin{example}  If $a\in O$ then the linear function $x\mapsto a\cdot x$ descends to an endomorphism of $\la K/O,+\ra$. In the case that $a\in \mu$, the map descends to an endomorphism of $K/O$ with infinite kernel. Thus we obtain a definable locally constant,  surjective endomorphism onf $K/O$.\end{example}

For a $K$-differentiable $F:K^n\to K$, and $i=1,\ldots, n$, we let $F_{x_i}$  denote the partial derivative with respect to $x_i$, and let $\nabla F=(F_{x_1}, \ldots, F_{x_n})$.

We are going to use several times the following (see Exercise 3.2.19(2)):\emph{ If $F:K^n\to K$ is an $\CL$-definable partial function and $p\vdash \dom(F)$ is a generic type then the restriction of $F$ to $p(\CM)$ is monotone in each coordinate separately.}

\begin{lemma} \label{K on O 1} Assume that $F:K^n\to K$ is an $\CL(A)$-definable (partial) function which descends to $K/O$.
Assume that $p\vdash \dom(F)$ is a partial long $\CL_{conv}(A)$-type which implies a complete generic $\CL(A)$-type in $K^n$,

 Then for every
$a\models p$,  and $i=1,\ldots, n$, we have  $F_{x_i}(a)\in O$.
\end{lemma}
\proof
Because $p$ implies a generic type in $K^n$ the  function $F$ is smooth at every  $a\models p$.  Assume towards a contradiction that the conclusion of the lemma does not hold. Without loss of generality, assume that $F_{x_1} (a)> O$ for some $a=(a_1,\ldots, a_n)\models p$.

Using monotonicity of $F_{x_1}$ in the first coordinate,  we can find some $\alpha>O$ and an  interval $J\sub K$ of length greater than $1$  such that every element in $J\times \{\la a_2,\ldots, a_n\ra \}$ still realizes $p$, and for every $x\in J$, $F_{x_1}   (x,a_2,\ldots, a_n) >\alpha$. By Lagrange's Mean Value Theorem for o-minimal structures, for $x,x+1\in J$,  $F(x+1, a_2,\ldots, a_n)-F(x,a_2,\ldots,a_n)>\alpha$, contradicting the assumption that $F$ descends to $K/O$.\qed

\begin{lemma}\label {K on O 2} Assume that $G:K^n \to K$ is an $\CL(A)$-definable partial function, and $p\vdash \dom(G)$ is a long partial $\CL_{conv}(A)$-type which implies a complete generic $\CL(A)$-type in $K^n$. Assume also that for every $a\models p$,  $G(a)\in O$. Then for all $a\models p$ and all $i\in \{1,\ldots, n\}$, we have $G_{x_i} (a)\in \mu$.
\end{lemma}
\proof
We prove the result for $i=1$. By Lemma \ref{convex}, there is $c\in O$ such that for all $a\models p$ we have $|G(a)|\leq c$.
Assume towards contradiction that for some $a=(a_1,\ldots, a_n)\models p$ we have $|G_{x_1}(a)| >\mu$. Then we can find a positive $\al\in O\setminus \mu$, such that $|G_{x_1}(a)|>\al$.
Using the montonicity of $G_{
x_1}$ in the first coordinate, we may assume that for every positive $b\in O$, we have $G_{x_1}(a_1+b,a_2,\ldots, a_n)>\al$. Pick $b>2c/\al \in O$. Then, by Lagrange's mean value theorem in o-minimal structures,
$$|G(a_1+b,a_2,\ldots,a_n)-G(a_1,\ldots,a_n)|>|G_{x_1}(\xi, a_2,\ldots,a_n))|(2c/\al)$$ for $\xi\in (a_1,a_1+b)$. By our assumptions, $|G_{x_1}(\xi, a_2,\ldots,a_n)|>\al$. It follows that  $G(a_1+b,a_2,\ldots,a_n)-G(a_1,\ldots,a_n)|>2c$, in contradiction to the fact that  $G|(x)|\leq c$ for all $x\models p$.\qed \\

We conclude:

\begin{lemma}\label{K on O 3} Assume that $F:K^n\to K$ is a partial $\CL(A)$-definable function which descends to $K/O$. Let $p\vdash \dom(F)$ be a  long $\CL_{conv}(A)$-type which implies a complete generic $\CL(A)$-type. Then there are $a_1,\ldots, a_n\in O$, $b\in K$ and a long box $B\sub p(K)$ such that for all $x= \la x_1,\ldots,x_n\ra \in B$,
	$$F(x)-\sum_{i=1}^n a_i(x_i-\alpha_i)-b\in \mu,$$.\end{lemma}
\proof By Lemma \ref{K on O 1}, each $F_{x_i}$ takes values in $O$, on the type $p$. By Lemma \ref{K on O 2}, applied to the each of the functions $F_{x_i}$, we have $F_{x_i,x_j}(a)\in \mu$ for all
$a\models p$ and $i,j=1,\ldots, n$. By Lemma \ref{convex}, there is some fixed positive $\beta\in \mu$ such that $|F_{x_i,x_j}(a)|< \beta$ for all $a\models p$.

Fix $\al\models p$ and let $(a_1,\ldots, a_n)=\nabla F(\al)$.

By \cite[Lemma 7.2.9 ]{vdDries}, for all $x\models p$,

\begin{equation} \label{TT}
\left | F(x)-F(\al)-\sum_{i=1}^n a_i(x_i-\alpha_i)\right | \leq |x-\al|\cdot \max_{y\in [\al,x]}|\nabla F(x)-\nabla F(\alpha)|, \tag{*}
\end{equation}
where $[\al,x]$ is the closed line segment in $K^n$ connecting $\al$ and $x$ and $|\nabla F(x)|$ is the operator norm.

Fix an $x\models p$ and for  $i=1,\ldots, n$, and $t\in [0,1]$, consider the function $$g_i(t)=F_{x_i}((1-t)\al+tx).$$  The derivative of $\frac{dg_i}{dt}$ is, by the chain rule,
$$(F_{x_1,x_i}(y), \ldots, F_{x_n,x_i}(y))\cdot (x-\al).$$ Applying  Lagrange's Theorem to  $g_i$ and substituting, we get $\xi \in (0,1)$
such that
$$|F_{x_i}(x)-F_{x_i}(\al)|=|g_i(1)-g_i(0)|=|g_i'(\xi)|=|(F_{x_1,x_i}(y), \ldots, F_{x_n,x_i}(y))\cdot (x-\al)|,$$
where $y=(1-\xi)\al+\xi x$.

Since each $|F_{x_j,x_i}(y)|<\beta$
%in absolute value is bounded by $\beta$,
we have $|F_{x_i}(x)-F_{x_i}(\al)|\leq \beta|x-\al|$ for all $i=1,\ldots, n$.
Thus $|\nabla F(x)-\nabla F(\al)|\leq n \beta |x-\al|$, so by (*),
$$| F(x)-F(\al)-\sum_{i=1}^n a_i(x_i-\alpha_i) |\leq n \beta |x-\al|^2.$$

Since $p$ is long and $\beta\in \mu$, we can  find a long box $B\sub p(\CM)$, centered at $\al$ such that  for every $x\in B$, we have $|x-\al|\leq (1/\beta)^{1/4}$.
Thus, for $x\in B$ we have $$| F(x)-F(\al)-\sum_{i=1}^n a_i(x_i-\alpha_i)|\leq n\sqrt{\beta}\in \mu.$$ If we now let $b=F(\al)$ then we have the desired result.
\qed \\

In fact, for the next corollary it would have been sufficient to show above that $F(x)-F(\al)-\sum_{i=1}^n a_i(x_i-\alpha_i)+b\in O$.

\begin{cor}\label{K on O- functions} Assume that $f:(K/O)^n\to K/O$ is a (partial) $\0$-definable function whose domain is open. Then for every generic $c\in \dom(f)$ there exists an open box $R\sub \dom(f)$ centered at $c$, definable endomorphisms of $\la K/O;+\ra $ denoted by   $\ell_1,\ldots\ell_n $, and $d\in K/O$, such that
	for every $y=\la y_1,\ldots, y_n\ra \in R$ we have
	$$f(y)=\sum_{i=1}^n\ell_i(y)+d.$$
	
\end{cor}
\proof Let $p=\pi^{-1}(tp_{K/O}(c/\0))$. It is clearly a long type and by Proposition \ref{opaque1}, it is in fact a complete $\CL_{conv}$-type which implies a generic $\CL$-type. Using definable Skolem functions we may lift $f$ to an $\CL$-definable partial $F:K^n\to K$, which descends to $K/O$. Namely, for every $x\in \dom(F)$, we have $F(x)\in f(x+O^n)$. We have $p\vdash \dom(F)$.

Applying Lemma \ref{K on O 3}, we can find a long box $B\sub K^n$ centered at $\al\models p$, and $a_1,\ldots, a_n\in O$,  $b\in K$, such that for every $x\in B$ we have
$$F(x)-b-\sum_{i=1}^n a_i(x_i-\alpha_i)\in \mu.$$

Each function $x\mapsto a_i\cdot x$ descends to an endomorphism $\ell_i$ of $\la K/O;+\ra $, so if we let $R$ be the image of $B$ in $K/O$ and  $e=b+O$,  then we have for all $y\in R$,
$$f(y)=\sum_{i=1}^n \ell_i(y_i-\alpha_i)+e=\sum_{i=1}^n \ell_i(y_i)+d, $$ for $d=e-\sum_{i=1}^n\ell_i(\al_i).$ \qed

We end this section by commenting that in the \emph{o-minimal setting} the local affiness of definable functions provided by Corollary \ref{K on O- functions},  would imply that the structure is linear (in the sense of \cite{LoPe}), and thus does not interpret a field. This is not true for $K/O$:

\begin{proposition}\label{k and K/O} The field $k$ is definably isomorphic to a field interpretable (in induced structure on) $K/O$. \end{proposition}
\proof   Fix some $t$ such that $v(t)<0$. Consider the balls $tO\supseteq t\mu$. Since $O\subseteq t\mu$,  the subgroups $tO/O$ and $t\mu/O$ are definable subgroups of $K/O$. Thus,  we get that the quotient $\la tO/t\mu;+\ra $ is interpretable in $K/O$. The former definably isomorphic to $\la k;+\ra$, hence we can definably endow it with multiplication which makes it isomorphic to $\la k;+,\cdot\ra$.\qed

As we shall show in Proposition \ref{Not K/O},  no field is \emph{definable} in $K/O$.

\section{Subsets of $\CF$ that are strongly internal to $I$}

	 As we saw in Theorem \ref{K/E},  if $T$ is power bounded then there exists an infinite $\CL_{conv}$-definable $I\subseteq \CF$ such that $I$ is in definable bijection with an interval in one of the structures $K$, $\Gamma$, $k$ or $K/O$. Once we obtain such a set $I$ we can drop the assumption that $T$ is power bounded. In all four cases,  the induced structure on $I$ is weakly o-minimal. Let us be more precise about this:

\begin{remark}
	Recall that given an $\aleph_1$-saturated structure $\CM$ a definable set $S\subseteq M^n$ is \emph{stably embedded} if for any definable $X\subseteq M^{nk}$ the set $X\cap S^k$ is definable using only parameters from $S$. Thus the $\CM$-induced structure on $S$ does not depend on the model or on the choice of parameters.
	
	However, intervals in a weakly o-minimal structure need not be stably embedded so that the notion of ``the structure induced on $I$'' is not well defined. For $A\sub K$ we let $\CI_A$ denote the structure obtained by taking the traces on $I^n$, $n\in \mathbb N$, of all $A$-definable sets. The structure $\mathcal I_K$  has a weakly o-minimal theory. So the same is true of any ordered reduct of the full $K$-induced structure on $I$ and in particular of reducts of the form $\mathcal I_A$.
 We assume that $\CF$, $I$, and an embedding of $I$ into $K, k, \Gamma$ or $K/O$ are all definable over $\0$, and initially work in $\mathcal I_\0$. We will, throughout the course of the proof, add small sets of parameters (not necessarily coming from $I$), tacitly expanding the structure $\mathcal I_\0$ in into one of the form $\mathcal I_A$.
We denote the resulted structure by $\CI$.
\end{remark}

	\begin{defnn} A definable $Y\sub \CF$ is \emph{strongly internal to $I$} (or {\em strongly $I$-internal}) if there exists a definable injection, possibly over additional parameters, from $Y$ into $I^r$ for some $r$.
	\end{defnn}
	
	\emph{We now fix $Y\sub \CF$ which is strongly internal to $I$ and has maximal dimension, denoted by $n$, in the sense of $I$ (equivalently, maximal dp-rank). }\\

We identify $Y$ with its image in $I^r$ and view it  with the induced $I$-topology.  We assume for simplicity that $Y$ is definable over $\0$. Unless otherwise stated we use $+,-,\cdot, ()^{-1}$ for the field operations of $\CF$.

	\begin{proposition}\label{prop0} Let $\la b,c,d\ra \in I\times Y^2$ be such that $\dim(b,c,d/\0)=2n+1$.
		Then there exists an initial segment $J\subseteq I^{>b}$ and a definable  $S\sub Y^2$, all defined over an additional parameter set  $B$, such that $\dim(c,d/B)=2n$, $\la c,d\ra  \in S$, (in particular, $\dim(S)=2n$), and  such that for every $\la x,y,z\ra\in J\times S\sub J\times Y^2$, we have $(x-b)y+z\in Y$.
	\end{proposition}
	\begin{proof} For $\la x,y, z\ra \in I\times Y\times Y$ consider the function $f_y(x,z)=xy-z$.

\vspace{.2cm}

\noindent{\bf Claim} The point $\la b,d\ra \sub I\times Y\sub I^{1+r}$ is not an isolated point of the set $f_c^{-1}(f_c(b,d))$.
\vspace{.2cm}

\proof Assume $\la b,d\ra$ were isolated in $f_c^{-1}(f_c(b,d))$. By the weak o-minimality of $Th(\CI)$, there is a number $r$ such that for every $d\in \mathrm{Im}(f_c)$, the set $f_c^{-1}(d)$ has at most $r$-many isolated points. Because $I$ is linearly ordered, each of these points is in $\dcl(c,d)$. Thus, there is a $c$-definable (partial) function $\sigma:\CF\to I\times Y$, such that for every $w\in \dom(\sigma)$,
$\sigma(w)$ is an isolated point of $f_c^{-1}(w)$, and if $w=f_c(b,d)$ then $\sigma(w)=\la b,d\ra $. In particular, $f_c\circ \sigma (w)=w$, so $\sigma$ is injective.

The image of $\sigma$ is a $c$-definable set $T$ containing $\la b,d\ra$ on which $f_c$ is injective. It follows that $\dim(T)=n+1$. Because $I$ and $Y$ were strongly $I$-internal then so is $f_c(T)$, and we have $\dim(f_c(T))=n+1$.  This contradicts the maximality of dimension of $Y$.\qed

Thus $\la b,d\ra$ is a cluster point of $f_c^{-1}(bc-d)$.	It now follows that there exists an open interval $J=(b,b')\sub I$ (or $(b',b)$) such that for every $x\in J$ there is $z\in Y$ with $xc-z=bc-d$. In particular, the map $x\mapsto xc-(bc-d)$ sends $J$ injectively into $Y$.  We now apply Lemma \ref{dp2} and obtain $b_1$, $b<b_1<b'$, such that $\dim(b_1cd/b)=1+2n$. In particular, $\dim(cd/b_1b)=2n$
	
	Our assumption is that for every $x$ in the interval $J_1=(b,b_1)$, we have $(x-b)c+d\in Y$. This is a first order property of $c,d$ over the parameters $bb_1$. Thus there is a $bb_1$-definable set $S\sub Y^2$, with $\la c,d\ra \in S$, such that for every $\la y,z\ra \in S$ and $x\in J_1$, we have $(x-b)y+z\in Y$. It follows that  $\dim(S)=2n$ and $\la c,d\ra $ is a generic point in $S$ over $bb_1$. We now replace $J$ with $J_1$. \end{proof}

\begin{proposition} \label{Not K/O}
	The set $I$ is neither a subset of $K/O$ nor a subset of $\Gamma$.
\end{proposition}
	
\proof  Assume for a contradiction that $I$ is either a subset of $K/O$ or of $\Gamma$, and denote these by $V$. In both cases $V$ has an underlying ordered group structures (which is in fact an ordered  $\mathbb Q$-vector space). In order to distinguish it from the additive structure of the field we denote here the additive structure on $V$ by $\la V;\oplus\ra$. We still use $+,\cdot$ to denote the operations of $\CF$.

 Let $\la b,c,d\ra \in I\times Y^2$ be such that $\dim(b,c,d/\0)=2n+1$ We fix $J\sub I$ and $S\sub Y^2$ as provided by Proposition \ref{prop0}. For simplicity we assume that all these sets are definable over $\0$. For $\la x,y,z \ra \in J\times S$, we let $F(x,y,z)=(x-b)y+z$. Our choice of $S$ and $J$ assures that $F$ takes values in $Y$.

Using cell decomposition in either $K/O$ or $\Gamma$ we may assume that $S$ is an  open subset of $I^s$. \\

\noindent\textbf{Claim.} There exists an additive (with respect to $\oplus$) function $L:V^{1+2n}\to V^s$ and $e\in V$, such that $F(x,y,z)=L(x,y,z)+e$ on a definable subset of $J\times S$ of the same dimension. \\
\proof The claim follows from Corollary \ref{K on O- functions} if $V$ is $K/O$. If $T$ is exponential then by Lemma \ref{exp}, $\Gamma$ is definably isomorphic to $K/O$, so the claim  holds for $\Gamma$ as well. If $T$ is power bounded then by Fact \ref{factt}, $\Gamma$ with its induced structure is an o-minimal vector space over the field of definable exponents. In that case the result follows from quantifier elimination for ordered vector spaces. \qed$_{\text{Claim}}$ \\

%the corresponding result in $\Gamma$ (using quantifier elimination for ordered vector spaces), we can find an additive (with respect to $\oplus$) function
%$L:V^{1+2n}\to V^s$ and $e\in V$, such that $F(x,y,z)=L(x,y,z)+e$ on a definable subset of $J\times S$ of the same dimension.

For simplicity of notation we still denote the set provided by the last claim as  $J\times S$.

By the definition of the function $F$ in the field $\CF$, for every $\la y_1,z_1\ra \neq \la y_2,z_2\ra $, there is at most one $x$ such that $F(x,y_1,z_1)=F(x,y_2,z_2)$.

Fix $x_0\in J\sub V$ such that $x_0\neq 0_V$.  The map $\la y,z\ra \mapsto F(x_0,y,z)$ sends $Z_1\times Z_2$ into $Y$. Since $\dim(Z_1\times Z_2)=2n>n$, it follows that the map cannot be injective and hence there are $\la y_1,z_1\ra\neq \la y_2,z_2\ra$ for which $F(x_0,y_1,z_1)=F(x_0,y_2,z_2)$.

It follows that  $L(x_0,y_1,z_1)=L(x_0,y_2,z_2)$, hence $L(x_0,y_1,z_1)\ominus L(x_0,y_2,z_2)=0$. The function $T: x\mapsto L(x,y_1,z_1)\ominus L(x,y_2,z_2):V\to V^s$ is an additive function from $V$ into $V^s$, and $x_0\in \ker(T)$. It follows that $\mathbb Z x_0\sub \ker(T)$, so by weak o-minimality, $\ker(T)$ is a convex subgroup of $(V,\oplus)$. Thus $T$ is locally constant. This implies that $F(x,y_1,z_1)=F(x,y_2,z_2)$ for infinitely many $x$, a contradiction.\qed \\

From now on we assume that $I\subseteq K$ or $I\subseteq k$.

\section{Fields interpretable in $T_{conv}$}

We fix  $Y$  a definable subset of $\CF$ strongly internal to $I$, a subset of $K$ or $k$,  and of maximal $\CI$-dimension.  As in the previous section, we identify $Y$ with a subset of $I^s$, and equip it with the topology induced from the weakly o-minimal topology of $I^s$. For simplicity we assume that $I$ and $Y$ are $\0$-definable. As before, the argument below does not require $T$ to be power bounded.

\subsection{Infinitesimal neighborhoods}

We introduce the notion of infinitesimals with respect to the ambient weakly o-minimal structures ($K$, or $k$):

Given $a\in I$, {\em the infinitesimal neighborhood of $a$} (with respect to the ambient weakly o-minimal structure $\CI$) is the partial type over $\CM$ consisting of all $\CM$-definable $\CI$-open sets  containing $a$. Equivalently, it is the collection $\{c<x<d: c<a<d, \,\, c,d\in I\}$. We denote it by $\nu(a)$. For $a=(a_1,\ldots, a_n)\in I^n$, we let $\nu(a)=\nu(a_1)\times\cdots \times \nu(a_n)$.
Note that $\nu(a)$ is a partial type in the language  $<$ (in either $K$ or $k$).

For a definable $Y\sub I^s$, and $a\in Y$, we let $\nu_Y(a)=Y\cap \nu(a)$.
In the case where  $I\sub k$,  $Y$ is definable in the induced o-minimal language on $k$ and therefore the type $\nu_Y(a)$ is given by $\CL_{ind}$-formulas.
As we now note, this is also the case when $I\sub K$:

\begin{remark}\label{mu-RCF}
	Assume that $I\sub K$, $Y\sub I^s$ and $a$ is generic in $Y$. By Lemma \ref{TconvQE}, there is an $\CL$-definable set $T$
and an open box $B\ni a$ such that $B\cap Y$ is $\CL$-definable. Thus, the type $\nu_Y(a)$ is given by $\CL$-formulas.
\end{remark}

Finally, note that if $f:I^r\to I^s$ is an $\CI$-definable function  continuous at  $a\in I^r$,  then $f(\nu(a))\sub \nu(f(a))$.

%. By that we mean $\sigma^{-1}(\nu(\sigma(a))\cap \sigma(Y))$, where $\nu(\sigma(a))$ is the partial type of the  infinitesimal neighborhood of $\sigma(a)$ inside $I^n$
%the intersection of all $K$-definable $I$-open sets containing $a$.

%To simply notation we identify $I$ with $\sigma(I)$ (over the additional) and thus $\nu_Y(a)=\nu(a)\cap Y$.

\subsection{The infinitesimal subgroup of $\la \CF,+\ra$ associated to $Y$}

Throughout  $+$, $-$, $\cdot$ and $(\, )^{-1}$ denote the operations in  $\CF$.

\begin{lemma}\label{infinitesimal} If $d$ is generic in $Y$ then $\nu_Y(d)-d$ is a type definable subgroup of $(\CF,+)$.
	Moreover, for every $d_1,d_2$ both generic in $Y$ over $\0$, we have $\nu_Y(d_1)-d_1=\nu_Y(d_2)-d_2$.
\end{lemma}
\proof  We fix any $b\in I$ and $c\in Y$ such that $\dim(b,c,d)=2n+1$ and apply Proposition \ref{prop0}.
It follows that the $b$-definable set
$$\{\la x,y,z\ra \in I\times Y\times Y:(x-b)y+z\in Y\}$$
has dimension $2n+1$ and contains a set of the form $J\times S$ with $J$ an interval $(b,b')$, $\dim(S)=2\dim Y$
and $\la c,d\ra$ generic in $S$ over $bb'$.

Apply Lemma \ref{dp2} to obtain $b_1\in J$ such that $\dim(b_1, c,d/bb')=2n+1$.
%In particular $\la b_1,c,d\ra$ is generic in  $J\times S$. Since $\CI$ is weakly o-minimal it follows (see \cite[Theorem 4.8, Theorem %4.11]{MacMaSt}) that  the $b$-definable function $\la x,y,z\ra \mapsto (x-b)y+z$ is continuous at $\la b_1,c,d\ra$ as a function from % $I^{1+2s}$ into $I^s$.
Let $Y'=(b_1-b)Y$ and $c'=(b_1-b)c$. Note that $Y'$ is strongly internal to $I$  and that $c'$ is inter-definable with $c$ over $b_1,b$. Therefore $\dim(c',d/b_1,b)=2n$, i.e.,  $\la c',d\ra $ is generic in $Y'\times Y$ over $b_1b$. Since $\CI$ is weakly o-minimal it follows (see \cite[Theorem 4.8, Theorem 4.11]{MacMaSt}) that  the function $(y,z)\mapsto y+z$, from $Y'\times Y\sub I^{2s}$ into $Y\sub I^s$, is $I$-continuous at $\la c',d\ra $.
In particular, it sends the types $\nu_{Y'}(c')\times \nu_Y(d)$ into $\nu_Y(c'+d)$.

\emph{We now work in an elementary extension of $\CM$ and realize the various infinitesimal types there}.

The function $y+z$ is injective in each coordinate and therefore for every $y\in\nu_{Y'}(c')$ the map $z\mapsto y+z$ induces a bijection between $\nu_Y(d)$ and $\nu_Y(c'+d)$. Similarly,   for every $z\in \nu_Y(d)$, the function $x\mapsto x+z$ induces a bijection of $\nu_Y(c')$ and $\nu_Y(c'+d)$. Finally,  for every $w\in \nu_Y(c'+d)$ the function $y\mapsto w-y$ induces a bijection between $\nu_{Y'}(c')$ and
$\nu_Y(d)$.

This shows that for every $y\in \nu_{Y'}(c')$ and $w\in \nu_Y(c'+d)$ there exists $z\in \nu_{Y}(d)$ such that $w=y+z$. Similarly we obtain that $\nu_{Y'}(c')-c'=\nu_Y(d)-d=\nu_Y(c'+d)-(c'+d)$ and they are all subgroups of $(\CF,+)$. E.g., if $x,y\in \nu(d)-d$ then $x+c'\in \nu_Y(c')$, $y+d\in \nu_Y(d)$ and therefore $(x+y+c'+d)\in \nu_Y(c'+d)$, so that $x+y\in \nu_Y(c'+d)=\nu_Y(d)$. Closure of $\nu_Y(d)$ under inverses is proved similarly.

In particular, the above shows that whenever $\la c,d\ra $ is generic in $Y\times Y$  then $\nu_Y(c)-c=\nu_Y(d)-d$. Since either $Y\subseteq K^s$ or $Y\subseteq k^s$ we know that $Y$ is definable in a weakly o-minimal structure with definable Skolem functions. Therefore, given any  $d_1,d_2\in Y$, each generic in $Y$, we can find, using Corollary \ref{dp3} $c\in Y$ such that both $\la c,d_1\ra $ and $\la c,d_2\ra $ are generic in $Y\times Y$. It follows that $\nu_Y(d_1)-d_1=\nu_Y(c)-c=\nu_Y(d_2)-d_2$. \qed \\

We may thus associate to every strongly $I$-internal $Y$ of maximal dimension a type-definable subgroup $\nu_Y$ of $(\CF,+)$, which is any of $\nu_Y(c)-c$ for $c$ generic in $Y$.
%Since $\la b,c,d\ra$ was arbitrary generic with respect to a strongly internal set $Y$ of maximal dimension, and
We now conclude:
\begin{lemma}\label{all nu the same} For every $Y_1,Y_2$ strongly internal to $I$ of maximal dimension, we have $\nu_{Y_1}=\nu_{Y_2}$, as partial types in $\CF$.\end{lemma}
\proof Let $Y$ be the disjoint union of $Y_1$ and $Y_2$ (both embedded in some $I^r$). Then $Y$ is strongly internal to $I$ and thus for every generic of $Y_1$ or of $Y_2$ can be viewed as a  generic of $Y$. The result follows from Lemma \ref{infinitesimal}.\qed \\

We let $\nu_{\CI}$ denote the type-definable subgroup $\nu_Y$ of $(\CF,+)$ associated with (any) strongly $I$-internal set $Y\sub F$ of maximal dimension.
We still work with a fixed such $Y$, but now replace it with $Y-a$ for $a$ generic in $Y$, so that $\nu_{\CI}=\nu_Y(0)$.
Because $\nu_Y(0)$ is an additive subgroup of $\CF$, there is a definable $Y'\sub Y$, with $\nu_Y(0)\vdash Y'$, such that $-Y'\sub Y$ and $Y'+Y'\sub Y$.
We replace $Y$ with $Y'\cap -Y'$, so now $-Y=Y$ and $Y+Y$ is still $\CI$-internal.

Our next goal is to show that the field $\CF$ equals to the set $Y\cdot (Y\setminus \{0\})^{-1}$ (thus interpretable in  $\CI$) and that in fact, it can be made  definable in the o-minimal language of either $K$ or $k$ (without the valuation). For that purpose we apply several times ideas from
 Marikova's work,  \cite{Marikova-groups}. We start with a direct analogue of Lemma 2.10 in her article.

\begin{lemma}\label{jana} Assume that $a\in Y$ is generic over the parameters defining $Y$. Then the functions $(x,y,z)\mapsto x-y+z$ and $(x,y,z)\mapsto -x+y-z$
are $\CI$-continuous at $(a,a,a)$. Furthermore, if $I\sub K$ then the functions are  $\CL$-definable  in a neighborhood of $(a,a,a)$.
\end{lemma}
\proof Fix $c\in Y$ such that $\la c,a\ra $ is generic in $Y^2$. By our assumptions on $Y$, $c+a,c-a,-a$ and $-c$ are all in $I^s$.
We view the map $x-y+z$ via a composition of $4$ maps, and use repeatedly the weakly o-minimal fact that definable functions are continuous at generic points of their domains and, in the case that $I\sub K$, use  Lemma \ref{conv-L}(4) which says that $\CL_{conv}$-
definable functions coincide with $\CL$-definable functions on neighborhoods of generic points of their domain.
$$\begin{array}{lll}
\phi_1 & : & (x,y,z)\mapsto (c+x,-y,z):Y^3\to Y^3,\\
\phi_2 & : &  (x,y,z)\mapsto (x+y,z):Y^3\to Y^2,\\
\phi_3 & : & (x,y)\mapsto (-c,x+y):Y^2 \to Y\times Y,\\
\phi_4 & : & (x,y)\mapsto x+y : Y\times Y \to Y.
\end{array}
$$
Since $a$ is generic in $Y$,  the function $x\mapsto -x$ is continuous near $a$ and since $\dim(c,a)=2n$, the function $x\mapsto c+x$ is continuous near $a$. Thus, $\phi_1$ is continuous near $(a,a,a)$ and sends $(a,a,a)$ to $(c+a,-a,a)$. We now have $(c+a,-a)$ generic in $Y^2$ thus $\phi_2$ is continuous near $(c+a,-a,a)$ and sends it to $(c,a)$.

We proceed in the same way, and at each step use the fact that we work near a generic point, to conclude that $\phi_i$ is continuous locally.
Finally, $x-y+z=\phi_4\phi_3\phi_2\phi_1(x,y,z)$ is continuous near $(a,a,a)$.

In the case that $I$ is a subset of $K$, the same argument shows that $x-y+z$ and $-x+y-z$ is $\CL$-definable near $(a,a,a)$. Applying this result to the point $(-a,-a,-a)$ we get the same result for $-x+y-z$.
\qed \\

We now proceed with our proof:

\begin{proposition}\label{subring}
\begin{enumerate}
\item  For every nonzero $c\in \CF(\CM)$, the partial types $c\cdot \nu_{\CI}$ and $\nu_{\CI}$ are equal.
\item
	The set of realizations of $\nu_{\CI}$ in an elementary extension is a subring of $\CF$.

\end{enumerate}\end{proposition}
\begin{proof}
To see (1), notice that for every nonzero $c\in \CF$, the set $cY$ is also strongly internal, of the same dimension, hence by Lemma \ref{all nu the same}, $c\nu_Y=\nu_{cY}=\nu_Y$, (as partial types).\\

(2). We first argue in $\CM$ and fix $a$ generic in $Y$. Since $a\nu_Y=\nu_Y$, there is, by compactness, a definable subset $Y_1\sub Y$ such $aY_1\sub Y$.
We fix $b\in Y_1$ such that $\la a,b\ra$ is generic in $Y^2$. Since $\nu_Y(a)-a$ does not depend on the choice of $a$ it suffices to show that  $(\nu_Y(a)-a) (\nu_Y(b)-b) =\nu_Y(ab)-ab$.
	
We now work in an elementary extension of $\CM$, with realizations of the various partial types.
 For any $x\in \nu_Y(a)$ and $y \in \nu_Y(b)$ we have $(x-a)(y-b)= xy-xb-ay+ab$. Since $\la a,b\ra$ is generic in $Y^2$, $\CF$-multiplication is continuous near $\la a,b\ra$, hence $xb, ay, xy\in \nu_Y(ab)$. Since $a,b$ are independent generic also $ab$ is generic in $Y$. Thus by  Lemma \ref{jana}, $xy-xb+ab\in \nu_Y(-ab)$. Thus, $(x-a)(y-b) \in  \nu_Y(-ab)+ab=\nu(\CI)$.
	\end{proof}

We can now conclude the proof of our main theorem:
\begin{theorem}\label{maintext} The field $\CF$ is definably isomorphic to a definable field in $k_{ind}$ or to an $\CL$-definable field in $K$.
 \end{theorem}
\proof Let $\nu_{\CI}=\nu_Y(0)$ be as above. By Proposition \ref{subring}, the partial type $\nu_{\CI}$ is invariant under mutiplication by scalars from $\CF$.  For every $c\in \CF$, the set $c Y\cap Y $ is infinite (as both sets contain $\nu_{\CI}$). Take a nonzero $s\in c Y\cap Y$ and then $s=cr$ for some nonzero $r\in Y$ and so $c=s(r^{-1})$. It follows that  \begin{equation}\label{CF} \CF=\{xy^{-1}:x,y\in Y \, y\neq 0\}=Y (Y\setminus \{0\})^{-1}.\end{equation}
In fact, since $\nu_{Y}(a)-a=\nu_{\CI}$ for every generic $a\in Y$, the above shows that for two generic elements $a,b\in Y$,  we have
$\CF=(Y-a)\cdot ((Y\setminus \{b\})-b)^{-1}$. Moreover, we may replace $Y$ by any relatively open neighborhoods $U,V\sub Y$ of $a$ and $b$ respectively,
and then $\CF=(U-a)(V^*-b)^{-1}$, where $V^*=V\setminus \{b\}$.

Because $Y$ is strongly internal to $I$, the above already implies that  $\CF$ is interpretable in $\CI$.

%{\em We note that this gives an alternative proof of the fact that $I\not\subseteq \Gamma$, as $\Gamma$ is a linear o-minimal structure, not interpreting any infinite field.}

Our standing assumption is: $I\subseteq K$ or $I\subseteq k$. When $I\sub k$ we can use elimination of imaginaries in the o-minimal structure $k_{ind}$ to deduce that $\CF$ is definable in $k_{ind}$. So we are left with the case $I\sub K$, and we want to prove that $\CF$ is definably isomorphic to an $\CL$-definable field in $K$.

We fix $\la a,b\ra $ generic in $Y^2$ such that $a\cdot b\in Y$. By (\ref{CF}) and the comment right after it,  for every $K$-neighborhoods $U\ni a$ and $V\ni b$, we have $(U-a)(V^*-b)^{-1}=\mathfrak F$. We define on $U\times V^*$ the following equivalence relation: $\la x_1,y_1\ra \sim \la x_2,y_2\ra$ if
$$(x_1-a)(y_1-b)^{-1}=(x_2-a)(y_2-b)^{-1} \,\, \Leftrightarrow (x_1-a)(y_2-b)=(x_2-a)(y_1-b).$$

Note that there is a definable bijection between $U\times V^*/\sim$ and $\CF$, given by $[\la x,y\ra]\mapsto (x-a)(y-b)^{-1}$.

\begin{claim}  There are $\CL$-definable $U\ni a$ and $V\ni b$ such that $\sim$ is $\CL$-definable on $U\times V^*$, and $U\times V^*/\sim$ has an $\CL$-definable set of representatives.\end{claim}
\proof Unravelling the definition of $\sim$, we obtain:
$$\la x_1,y_1\ra \sim \la x_2,y_2\ra \Leftrightarrow x_1y_2-ay_2-x_1b=x_2y_1-ay_1-x_2b.$$

The function $\la x,y\ra \mapsto xy$ is $\CL$-definable near $\la a,b\ra$ because of genericity. Thus, $y\mapsto ay$ is $\CL$-definable near $b$ and $x\mapsto xb$ is $\CL$-definable near $a$. Finally, by Lemma \ref{jana}, the map $\la x,y,z\ra\mapsto -x+y-z$ is $\CL$-definable near the point $(ab,ab,ab)$. It follows that the function $\la x,y\ra \mapsto xy-ay-xb$ is $\CL$-definable near $\la a,b\ra $, thus $\sim$ is an $\CL$-definable relation near $\la a,b\ra $.

We fix $\CL$-definable relatively open neighbourhoods $ U,V\sub Y$ of $a$ and $b$, respectively, such that the restriction of $\sim$ to $U\times V^*$ is $\CL$-definable. We require further, using Lemma \ref{jana}, that the functions $(x,y,z)\mapsto x- y + z$  and $(x,y,z)\mapsto -x+y-z$ are $\CL$-definable on $U\cdot V$.
Using  definable choice in o-minimal structures we can thus find an $\CL$-definable set of representatives, call it $S$, for $\sim$.\qed\\

{\em We have so far a definable bijection between $\mathfrak F$ and $S$, an $\CL$-definable set in $K^n$. Namely, $\CF$ is definably isomorphic to an $\CL_{conv}$-{\em definable} field (as the field operations might still be definable in $\CL_{conv}$).}\\

Our goal is to show that $\CF$ is definably isomorphic to an $\CL$-definable field.  One approach for doing that is by noticing that the proof of the analogous result, \cite[Theorem 4.2]{BaysPet}, in the case of real closed valued fields, works word-for-word for $T_{conv}$. However, for the sake of completeness we give a different,  self-contained proof  that we can endow $S$ with $\CL$-definable field operations making $\CF$ into a $\CL$-definable field.  We first need: \\

\begin{claim}\label{last} Given $\la x,y\ra \in (U\setminus \{a\})\times (V\setminus\{b\})$, for all $Y$-neighborhoods $U_1\ni a$ and $V_1\ni b $ sufficiently small,
	and for all $a_1\in U_1\setminus \{a\}$ and $b_1\in V_1\setminus \{b\}$ there are unique $x_1\in U_1$ and unique $y_1\in V_1$ such that
	$$\la x,y\ra \sim \la a_1,y_1\ra \sim \la x_1,b_1\ra.$$

Moreover, the two  families of functions $F_{x,y}(b_1)=x_1$ and $G_{x,y}(a_1)=y_1$ are $\CL$-definable, as $\la x,y\ra$ varies.
\end{claim}
\proof The uniqueness is clear from the definition. Let $c=(x-a)^{-1}(y-b)$. The sets $c(Y-a)$ and $(Y-b)$ both contain $\nu_{\CI}$. Thus, for $a_1\in \nu_{Y}(a)$ there exists $y_1\in \nu_Y(b)$ such that $c(a_1-a)=y_1-b$. It now follows that $(x-a)(y-b)^{-1}=(a_1-a)(y_1-b)^{-1}$, so $\la x,y\ra\sim \la a_1,y_1\ra$.

Similarly, there exists $x_1\in \nu_Y(a)$ such that $\la x,y\ra \sim \la x_1,b_1\ra$.  The existence of $U_1,V_1$ follows by compactness.
Because $S$ is $\CL$-definable, the map  $\la x,y, b_1\ra \mapsto x_1$ and the map $\la x,y,a_1\ra \mapsto y_1$ are $\CL$-definable on their appropriate domains.\qed

We can now show that the field operations, as induced from $\CF$ on the $\CL$-definable set $S$, are $\CL$-definable.
Given $U\ni a$ and $V\ni b$, we let $U^*=U\setminus\{a\}$ and $V^*=V\setminus \{b\}$.
\vspace{.2cm}

\noindent{\bf The definability of addition.}
\vspace{.2cm}

For $\la x_1,y_1\ra,\la x_2,y_2\ra, \la x_3,y_3\ra\in S$, let
$$M_{+} ((x_1,y_1),(x_2,y_2))=\la x_3,y_3\ra \mbox{ if } (x_1-a)(y_1-b)^{-1}+(x_2-a)(y_2-b)^{-1}=(x_3-a)(y_3-b)^{-1}.$$

If $x_1=a$ then $M_{+} ((x_1,y_1),(x_2,y_2))=\la x_2,y_2\ra$. Similarly, $M_{+} ((x_1,y_1),(a,y_2))=\la x_1,y_1\ra$.
We thus assume that $x_1\in U^*$ and $y_1\in V^*$.

By Claim \ref{last}, we can replace $\la x_2,y_2\ra$ with $\la F_{x_2,y_2}(y_1), y_1\ra$ and , $\la x_3,y_3\ra$ with $\la F_{x_3,y_3}(y_1), y_1\ra$, uniformly and $\CL$-definably in the parameters. So we may assume that $y_1=y_2=y_3$. Setting $x_i':=F_{x_i,y_i}(y_1)$ we get
%
%
%after applying the appropriate functions $F_{x,y}$ (and changing $x_1$ and $x_2$ to $x_1', x_2'$),  that $y_1=y_2=y_3$, and now
$$M_{+}((x_1,y_1),(x_2,y_1))=\la x_3,y_1\ra \Leftrightarrow  (x_1'-a)+(x_2'-a)=(x_3-a)\Leftrightarrow x_1'-x_3'+x_2=a.$$
By the choice of $U$,  we may apply Lemma \ref{jana}, to get that this is an $\CL$-definable relation.
\vspace{.2cm}

\noindent{\bf The definability of multiplication.}

\vspace{.2cm}

For $\la x_1,y_1\ra,\la x_2,y_2\ra, \la x_3,y_3\ra\in S$, let
$$M_{\bullet} ((x_1,y_1),(x_2,y_2))=\la x_3,y_3\ra \mbox{ if } (x_1-a)(y_1-b)^{-1}\cdot (x_2-a)(y_2-b)^{-1}=(x_3-a)(y_3-b)^{-1}.$$

To see that $M_{\bullet}$ is $\CL$-definable, we first fix
$a_0\in U^*$ and $b_0\in V^*$ such that $a_0-a=b_0-b$. Using $F$, as above, we can find $x_2'$ such that $\la x_2',y_3\ra \sim \la x_2,y_2\ra$. Thus

%Changing $x_2$ to $x_2'$ using $F$ above, we may assume that $y_2=y_3$ and thus
$$M_{\bullet} ((x_1,y_1),(x_2,y_2))=\la x_3,y_3\ra \Leftrightarrow (x_1-a)(y_1-b)^{-1}\cdot (x_2'-a)=x_3-a.$$

We can now change, $\CL$-definably,  $\la x_1,y_1\ra$ in a similar way to an equivalent  $\la x_3, y_1'\ra $ so that
$$M_{\bullet} ((x_1,y_1),(x_2,y_2))=\la x_3,y_3\ra \Leftrightarrow (y_1'-b)^{-1}\cdot (x_2'-a)=1.$$

We now use the fact that $(a_0-a)(b_0-b)^{-1}=1$ to conclude:
$$M_{\bullet} ((x_1,y_1),(x_2,y_2))=\la x_3,y_3\ra \Leftrightarrow \la x_2',y_1'\ra\sim \la a_0,b_0\ra.$$

Thus $M_{\bullet}$ is $\CL$-definable, concluding the proof of the theorem. \qed \\

For the sake of clarity we now sum up everything done up until this point to get a complete  proof of Theorem \ref{main}:

\begin{cor}\label{final}
	Let $T$ be an o-minimal power bounded expansion of a field, $T_{conv}$ a $T$-convex expansion of $T$. Any field $\CF$ which is interpretable in $K\models T_{conv}$  is definably isomorphic to one of $K$, $K(\sqrt{-1})$, $k$, or $k(\sqrt{-1})$
\end{cor}
\begin{proof}
	Fix some interval $J\subseteq K$ and a definable equivalence relation $E$ on $J$ such that $J/E$ is in definable bijection with an infinite subset of $\CF$, as provided by Lemma \ref{1-dim}. Shrink $J$ further to obtain a definable bijection between the weakly o-minimal structure $J/E$ and an infinite interval in $K$, in $k$, in $\Gamma$ or in $K/O$ (Theorem \ref{K/E}). Apply Proposition \ref{Not K/O} to eliminate the two last cases. So we are reduced to the case where $I$ is in definable bijection with either $k$ or $K$ and use Theorem \ref{maintext} to get that $\CF$ definable in either the o-minimal structure $k_{ind}$ or in the o-minimal structure $K$ (restricted to $\CL$). This allows us to use the main result of \cite{OtPePi} to complete the proof.
\end{proof}

%\noindent{\bf A final remark}: Theorem \ref{final} was proved under the assumption that the underlying o-minimal theory $T$ is power bounded. This was used in order to eliminate the possibility that $\CF$ was interpretable in  $\Gamma$. In case that $T$ is exponential, the same proof will show that the  field $\CF$ is definably isomorphic to  either an $\CL$-definable field in $K$ (so definably isomorphic to $K$ or $K(\sqrt{-1})$), or to a $k_{ind}$-definable field (so isomorphic to $k$ or $k(\sqrt{-1})$), or to a field interpretable in the structure which $\CM$ induces on $\Gamma$. In this last case, it is not clear what could be further said of $\CF$.


\begin{thebibliography}{10}
	
	\bibitem{BaysPet}
	Martin Bays and Ya'acov Peterzil.
	\newblock Definability in the group of infinitesimals of a compact {L}ie group.
	\newblock {\em Confluentes Math.}, 11(2):3--23, 2019.
	
	\bibitem{CherStarch}
	Artem Chernikov and Sergei Starchenko.
	\newblock Regularity lemma for distal structures.
	\newblock {\em J. Eur. Math. Soc. (JEMS)}, 20(10):2437--2466, 2018.
	
	\bibitem{EHM}
	Clifton Ealy, Deirdre Haskell, and Jana Ma\v{r}\'{\i}kov\'{a}.
	\newblock Residue field domination in real closed valued fields.
	\newblock {\em Notre Dame J. Form. Log.}, 60(3):333--351, 2019.
	
\bibitem{Pminimal}
Yarit Halevi, Assaf Hasson and Ya'acov Peterzil.
\newblock Fields interpretable in $P$-minimal fields.
\newblock{Math Arxiv, 2021}

	\bibitem{HaHrMac1}
	Deirdre Haskell, Ehud Hrushovski, and Dugald Macpherson.
	\newblock {Definable sets in algebraically closed valued fields: elimination of
		imaginaries}.
	\newblock {\em J. Reine Angew. Math.}, 597:175--236, 2006.
	
	\bibitem{HaHrMac2}
	Deirdre Haskell, Ehud Hrushovski, and Dugald Macpherson.
	\newblock {\em {Stable domination and independence in algebraically closed
			valued fields}}, volume~30 of {\em {Lecture Notes in Logic}}.
	\newblock Association for Symbolic Logic, Chicago, IL, 2008.
	
	\bibitem{HaHrMac3}
	Deirdre Haskell, Ehud Hrushovski, and Dugald Macpherson.
	\newblock Unexpected imaginaries in valued fields with analytic structure.
	\newblock {\em J. Symbolic Logic}, 78(2):523--542, 2013.
	
	\bibitem{Hempel-Palacin}
	Nadja Hempel and Daniel Palac\'{\i}n.
	\newblock Division rings with ranks.
	\newblock {\em Proc. Amer. Math. Soc.}, 146(2):803--817, 2018.
	
	\bibitem{HolEOI1}
	Jan~E. Holly.
	\newblock Prototypes for definable subsets of algebraically closed valued
	fields.
	\newblock {\em J. Symbolic Logic}, 62(4):1093--1141, 1997.
	
	\bibitem{HrRid}
	Ehud Hrushovski and Silvain Rideau-Kikuchi.
	\newblock Valued fields, metastable groups.
	\newblock {\em Selecta Math. (N.S.)}, 25(3):Paper No. 47, 58, 2019.
	
	\bibitem{johnsondpfiniteII}
	Will Johnson.
	\newblock Dp-finite fields ii: the canonical topology and its relation to
	henselianity, 2019.
	
	\bibitem{johnsondpfiniteI}
	Will Johnson.
	\newblock Dp-finite fields i: infinitesimals and positive characteristic, 2020.
	
	\bibitem{johnsondpfiniteVI}
	Will Johnson.
	\newblock Dp-finite fields vi: the dp-finite shelah conjecture, 2020.
	
	\bibitem{KOUadd}
	Itay Kaplan, Alf Onshuus, and Alexander Usvyatsov.
	\newblock Additivity of the dp-rank.
	\newblock {\em Trans. Amer. Math. Soc.}, 365(11):5783--5804, 2013.
	
	\bibitem{LoPe}
	James Loveys and Ya'acov Peterzil.
	\newblock Linear o-minimal structures.
	\newblock {\em Israel J. Math.}, 81(1-2):1--30, 1993.
	
	\bibitem{MacMaSt}
	Dugald Macpherson, David Marker, and Charles Steinhorn.
	\newblock {Weakly o-minimal structures and real closed fields}.
	\newblock {\em Trans. Amer. Math. Soc.}, 352(12):5435--5483 (electronic), 2000.
	
	\bibitem{Marikova-groups}
	Jana Ma\v{r}\'{\i}kov\'{a}.
	\newblock Type-definable and invariant groups in o-minimal structures.
	\newblock {\em J. Symbolic Logic}, 72(1):67--80, 2007.
	
	\bibitem{Mellor}
	T.~Mellor.
	\newblock Imaginaries in real closed valued fields.
	\newblock {\em Ann. Pure Appl. Logic}, 139(1-3):230--279, 2006.
	
	\bibitem{MilStar}
	Chris Miller and Sergei Starchenko.
	\newblock A growth dichotomy for o-minimal expansions of ordered groups.
	\newblock {\em Trans. Amer. Math. Soc.}, 350(9):3505--3521, 1998.
	
	\bibitem{OtPePi}
	Margarita Otero, Ya'acov Peterzil, and Anand Pillay.
	\newblock On groups and rings definable in o-minimal expansions of real closed
	fields.
	\newblock {\em Bull. London Math. Soc.}, 28(1):7--14, 1996.
	
	\bibitem{PePiSt}
	Ya'acov Peterzil, Anand Pillay, and Sergei Starchenko.
	\newblock Simple algebraic and semialgebraic groups over real closed fields.
	\newblock {\em Trans. Amer. Math. Soc.}, 352(10):4421--4450, 2000.
	
	\bibitem{PilQp}
	Anand Pillay.
	\newblock On fields definable in {${\bf Q}_p$}.
	\newblock {\em Arch. Math. Logic}, 29(1):1--7, 1989.
	
	\bibitem{PoiFields}
	Bruno Poizat.
	\newblock M{M}.\ {B}orel, {T}its, {Z}ilber et le {G}\'en\'eral {N}onsense.
	\newblock {\em J. Symbolic Logic}, 53(1):124--131, 1988.
	
	\bibitem{Simdpr}
	Pierre Simon.
	\newblock Dp-minimality: invariant types and dp-rank.
	\newblock {\em J. Symb. Log.}, 79(4):1025--1045, 2014.
	
	\bibitem{SiBook}
	Pierre Simon.
	\newblock {\em A guide to {NIP} theories}, volume~44 of {\em Lecture Notes in
		Logic}.
	\newblock Association for Symbolic Logic, Chicago, IL; Cambridge Scientific
	Publishers, Cambridge, 2015.
	
	\bibitem{SimWal}
	Pierre Simon and Erik Walsberg.
	\newblock Tame topology over dp-minimal structures.
	\newblock {\em Notre Dame J. Form. Log.}, 60(1):61--76, 2019.
	
\bibitem{Tyne}
James Michael Tyne.
\newblock $T$-levels and $T$-convexity
\newblock Ph.D. Thesis, Univerisyt of Illinois at Urbana-Champaign, 2003.


	\bibitem{vdDriesTconv}
	Lou van~den Dries.
	\newblock {$T$}-convexity and tame extensions. {II}.
	\newblock {\em J. Symbolic Logic}, 62(1):14--34, 1997.
	
	\bibitem{vdDriesCorr}
	Lou van~den Dries.
	\newblock Correction to: ``{$T$}-convexity and tame extensions. {II}'' [{J}.
	{S}ymbolic {L}ogic {\bf 62} (1997), no. 1, 14--34; {MR}1450511 (98h:03048)].
	\newblock {\em J. Symbolic Logic}, 63(4):1597, 1998.
	
	\bibitem{vdDries}
	Lou van~den Dries.
	\newblock {\em Tame topology and o-minimal structures}, volume 248 of {\em
		London Mathematical Society Lecture Note Series}.
	\newblock Cambridge University Press, Cambridge, 1998.
	
	\bibitem{vdDrLew}
	Lou van~den Dries and Adam~H. Lewenberg.
	\newblock {{$T$}-convexity and tame extensions}.
	\newblock {\em J. Symbolic Logic}, 60(1):74--102, 1995.
	
\end{thebibliography}
\end{document}